\documentclass[reqno]{amsart}
\usepackage{url}
\newtheorem{theorem}{Theorem}[section]

\newtheorem{corollary}[theorem]{Corollary}
\newtheorem{conjecture}[theorem]{Conjecture}

\theoremstyle{definition}
\newtheorem{definition}[theorem]{Definition}

\theoremstyle{remark}

\numberwithin{equation}{section}
\usepackage{algorithmic}
\usepackage[top=3cm,bottom=2cm,right=2cm,left=2cm]{geometry}
\usepackage{hyperref}
\usepackage{amsmath}
\usepackage{amssymb}
\allowdisplaybreaks

\begin{document}

\title[Congruence Families Modulo Powers of 5 for Generalized Frobenius Partition Functions]{Old Meets New: Connecting Two Infinite Families of Congruences Modulo Powers of 5 for Generalized Frobenius Partition Functions}


\author{Frank G. Garvan}
\address{Department of Mathematics, University of Florida, P.O. Box 118105, Gainesville, FL 32611-8105, USA}
\curraddr{}
\email{fgarvan@ufl.edu}
\thanks{}

\author{James A. Sellers}
\address{Department of Mathematics and Statistics, University of Minnesota Duluth, 1049 University Drive, Duluth, MN 55812, USA}
\curraddr{}
\email{jsellers@d.umn.edu}
\thanks{}

\author{Nicolas Allen Smoot}
\address{Faculty of Mathematics, University of Vienna, Oskar-Morgenstern-Platz 1, 1090 Vienna, Austria}
\curraddr{}
\email{nicolas.allen.smoot@univie.ac.at}
\thanks{}



\date{}

\dedicatory{}

\begin{abstract}
In 2012 Paule and Radu proved a difficult family of congruences modulo powers of 5 for Andrews' 2-colored generalized Frobenius partition function. The family is associated with the classical modular curve of level 20. We demonstrate the existence of a congruence family for a related generalized Frobenius partition function associated with the same curve. We construct an isomorphism between this new family and the original family of congruences via a mapping on the associated rings of modular functions.  The pairing of the congruence families provides a new strategy for future work on congruences associated with modular curves of composite level.  We show how a similar approach can be made to multiple other recent examples in the literature.  We also give some important insights into the behavior of these congruence families with respect to the Atkin--Lehner involution which proved very important in Paule and Radu's original proof.
\end{abstract}

\maketitle

\section{Introduction}

In his 1984 AMS Memoir, Andrews \cite{Andrews} introduced and extensively studied two families of generalized Frobenius partition functions.  The second such family of functions enumerates generalized $k$-colored Frobenius partitions of weight $n$, where the functions in question are denoted $c\phi_k(n)$.  In his Memoir, Andrews proved a variety of properties satisfied by these functions $c\phi_k(n)$; in particular, Andrews proved that, for all $n\geq 0$, $c\phi_2(5n+3) \equiv 0 \pmod{5}$.

In 1994 the second author proposed the following congruence family for the generalized 2-colored Frobenius partition function $c\phi_2(n)$:

\begin{theorem}[Paule, Radu]\label{Thm1}
Let $n,\alpha\in\mathbb{Z}_{\ge 1}$ such that $12n\equiv 1\pmod{5^{\alpha}}$.  Then $c\phi_2(n)\equiv 0\pmod{5^{\alpha}}$.
\end{theorem}  This congruence family deceptively resembles the classic congruences for the partition function $p(n)$ discovered by Ramanujan, but presents much greater challenges; a complete proof was not found until 2012.

Our principal discovery is that the congruence family of Theorem \ref{Thm1} is only one specific manifestation of a more general $\ell$-adic behavior on sequences of modular functions on the curve $\mathrm{X}_0(20)$.  In particular, this behavior manifests itself as a \textit{second} congruence family for the Fourier coefficients of a different modular form.

It appears that a remarkable property exists for congruence families associated with modular curves of composite level---in some sense, the \textit{same} family can manifest in the Fourier coefficients associated with \textit{multiple} modular forms.  We show how our approach may be applied to many other recently studied partition functions.

We also discuss some of the implications of this approach to the localization method of congruence families, and some peculiar properties that the functions associated with Theorem \ref{Thm1} have with respect to Ramanujan's phi function.

\subsection{Background}

What makes Theorem \ref{Thm1} so interesting is that, while it superficially resembles the classical congruence families of the unrestricted partition function $p(n)$ modulo powers of 5, 7, 11 that Ramanujan first discovered \cite{Ramanujan}, it is a much more difficult result to prove.  Indeed, a proof was not found until Paule and Radu's work in 2012 \cite{Paule}, 18 years after its proposal, and some 93 years after Ramanujan founded the subject.

This family poses problems for us still.  Like that of other congruence families, the proof given by Paule and Radu involves the $\ell$-adic properties of an associated sequence of meromorphic functions on a classical modular curve.  However, the proof that Paule and Radu provide invokes a remarkable algebraic machinery that is not a priori obvious or natural in the framework of the corresponding modular curve.

The third author has invented some more general techniques, embodied in the localization method, in an attempt to give a unified framework to the subject of partition congruences.  These techniques revolve around manipulation of certain rational polynomial sequences of modular functions to extract congruence properties.  This approach works in many different cases (e.g., \cite{Smoot}, \cite{Smoot0}, \cite{Baner7}), and indeed it reduces to standard methods for classical congruence families; however, this approach fails to resolve the  congruences of Theorem \ref{Thm1}.  As such, we cannot at present extend classical methods via a systematic approach to congruence families of a certain class which encompass Theorem \ref{Thm1} and other contemporary congruence families, some of which are yet standing conjectures.  We know that the family has been proved, but we still do not fully understand \textit{why} it is there, and why more systematic approaches fail to resolve it.

In this work we report on recent progress in studying the underlying difficulties of this family.  To our surprise, we have discovered that Theorem \ref{Thm1} is but part of a natural and extremely closely related \textit{pair} of congruence families, exhibited by a pair of Frobenius partition functions.  It may be that one reason that the congruences for $c\phi_2$ have proved so difficult to understand is that in the decade since they were resolved (and the three decades since they were first proposed), they have been effectively \textit{half} perceived.

This discovery arose after we became aware of a study of generalized Frobenius partition functions by Jiang, Rolen, and Woodbury \cite{Jiang}.  They define a class of functions of the form $c\psi_{k,\beta}(n)$ for which 
\begin{align}
c\psi_{2,1}(n) &= c\phi_2(n)
\end{align} for all $n\ge 0$.  Given the precise definition of this class of functions, as well as the restrictions on $k$ and $\beta$ (see below), it is natural to consider the function $c\psi_{2,0}(n)$.  Indeed, this function had been studied previously by Drake \cite{Drake}, who used the notation $c\psi_2(n)$.  We discovered that $c\psi_{2}(n)$ exhibits the following congruence family:

\begin{theorem}\label{Thm12}
Let $n,\alpha\in\mathbb{Z}_{\ge 1}$ such that $6n\equiv -1\pmod{5^{\alpha}}$.  Then $c\psi_{2}(n)\equiv 0\pmod{5^{\alpha}}$.
\end{theorem}  Henceforth we will use the notation developed by Jiang, Rolen, and Woodbury \cite{Jiang}, since it has proved so insightful; in particular, it allows us to express Theorems \ref{Thm1} and \ref{Thm12} together in the following form:

\begin{corollary}\label{Thm13}
Fix $\beta\in\{0,1\}$ and let $n,\alpha\in\mathbb{Z}_{\ge 1}$ such that $3\cdot 2^{\beta+1}n\equiv (-1)^{\beta+1}\pmod{5^{\alpha}}$.  Then
\begin{align*}
c\psi_{2,\beta}(n)\equiv 0\pmod{5^{\alpha}}.
\end{align*}
\end{corollary}

The latter corollary already suggests a close relationship between the two families.  However, we were surprised to discover just how close this connection is.  In the first place, each family is associated with a sequence of functions defined on the classical modular curve $\mathrm{X}_0(20)$---a curve of genus 1 and cusp count 6.  For each case, one constructs a pair of rank 2 $\mathbb{Z}[t]$-modules (with $t$ a natural Hauptmodul corresponding to $\mathrm{X}_0(5)$), and demonstrates membership of the respective functions within these modules.

For each family, the \textit{generators} of the respective modules are different; however, we have found that the proof of Theorem \ref{Thm12} is---at least in terms of \textit{formal} algebraic manipulations---\textit{identical} to that of Theorem \ref{Thm1}.  Indeed, the families are \textit{equivalent}; they are related to each other via a certain invertible mapping between subrings of the modular functions on $\mathrm{X}_0(20)$ which fixes the modular functions on $\mathrm{X}_0(5)$.  

It appears that in some strange sense Theorems \ref{Thm1} and \ref{Thm12} constitute a ``single" congruence family (i.e., a specific pattern of $5$-adic convergence to 0), which somehow manifests itself in \textit{two different} ways, for two different functions.  This approach provides a new strategy for the future study of congruence families associated with modular curves of composite level.

We discovered a second strange property which emerges when we examine the module generators associated with each respective congruence family.  Apart, each is a function defined on $\mathrm{X}_0(20)$.  When combined, they produce a function defined on $\mathrm{X}_0(10)$, a curve of genus 0 and cusp count 4.  This is especially interesting, in that congruence families associated with such curves, i.e., Riemann surfaces of simpler topologies, are much more accessible to our standardized proof methods.

Finally, these families bear a strange symmetry to one another upon the application of a certain Atkin--Lehner involution to their respective functions.  This symmetry gives rise to an absolutely astonishing and beautiful result which explains a curious coefficient behavior that was discovered by Paule and Radu in their original work \cite[Section 6]{Paule}, and used to assist in the construction of their proof of Theorem \ref{Thm1}.

The remainder of our paper is organized as follows: in Section \ref{sectdef} we give the definitions for the associated Frobenius partitions as well as the generating functions relevant to our work.  In Section \ref{mainthmetc} we give the underlying modular function spaces over $\mathrm{X}_0(20)$ that will be relevant to us.  We prove the initial stages of both congruence families via the associated initial terms on the related function sequences.  We then show how Theorem \ref{Thm12} can be derived from Theorem \ref{Thm1} via a certain mapping on a subring of modular functions of $\mathrm{X}_0(20)$.  We also give an explicit form for this mapping, namely as the action of a specific element of $\Gamma_0(10)$.

In Section \ref{sectaddit} we describe more examples of partition congruences in which a similar approach can be used to show that certain divisibility properties are shared between functions which at first sight appear very different.  The first is a pair of congruences exhibited by the Frobenius partition functions $\phi_2(n)$ and $\psi_2(n)$, studied by Andrews \cite{Andrews} and by Eckland and the second author \cite{Eckland}.  Individual congruences are rarer for these functions, and no infinite families are known to exist.  Our method suggests that any families for one of these functions would necessarily imply an associated family for the other.  Moreover, we show how a witness identity for a congruence for $\phi_2(n)$, which is derived by a comparatively tedious cusp analysis, can be used to derive (in a far less tedious manner) an analogous witness identity for the related congruence for $\psi_2(n)$.

The second example involves a pair of altogether different functions sudied by Connor Morrow in collaboration with the first author \cite{GarvanM}.  One is a congruence family exhibited by a function related to the crank parity function originally proved by Choi, Kang, and Lovejoy in \cite{Choi}.  The other is a family exhibited by a function closely related to the 1-shell TSPP and $\overline{\mathcal{EO}}(n)$ functions studied in \cite{Andrews2}, \cite{Chern2}, and \cite{Hirschhorn}.  Moreover, we note that yet more recent work in this subject has been done by the first author in concert with Dandan Chen and Rong Chen \cite{Chen2}.

Still another example involves a pair of infinite \textit{internal} congruence families exhibited by coefficients related to quotients of Ramanujan's theta functions, studied by Chern and Tang \cite{Chern}.

Note the variety of different functions in our examples.  We have Frobenius partitions, of course; but we also have objects related to mock theta functions, the partition crank statistic, and Ramanujan's classic theta functions.  Other potential examples that warrant investigation include smallest parts functions (especially those related to mock theta functions), $k$-elongated plane partitions, and the Hecke trace functions $t_m(d)$.

The key unifying property of all of these examples is that the modular curve associated with the congruence families has composite level.  The latter property also appears to be critical in studying the most difficult problems of the subject.

As such, a proper understanding of ``equivalent congruences" associated with composite level modular curves appears to be an extremely important phenomenon, though its universality has not hitherto been fully grasped.  We hope that we have effectively conveyed the significance of this approach to future work.

In Section \ref{sectreduc} we show how the two function sequences corresponding to Theorems \ref{Thm1} and \ref{Thm12}, when combined, produce functions which reduce to the simpler curve $\mathrm{X}_0(10)$.  This has implications for the more systematic programmes (e.g. localization) which have hitherto failed to account for congruence families like those of Theorems \ref{Thm1} and \ref{Thm12}.  We give conjectures as to the accessibility of proving the two congruence families \textit{relative to one another} via localization.

Finally, in Section \ref{sectinvol} we show how a peculiar coefficient behavior discovered by Paule and Radu is the consequence of a very striking relationship with Ramanujan's $\varphi$ function (i.e., Jacobi's classic theta function $\vartheta_3(0|2\tau)$).  In Section \ref{sectfin} we summarize and present the questions and conjectures which we consider especially important for further work.

In this work we presume an understanding of the theory of modular functions and its associated modular curves.  See \cite[Section 2]{Smoot}.  A useful classical reference is \cite{Knopp}, and a more modern reference is \cite[Chapters 2-3]{Diamond}.

We will include with this paper online access to a Mathematica notebook which will supplement the results of this paper, many of which were justified with computations that cannot be included in this article without substantially extending its size.  This notebook can be found at \href{https://www.d.umn.edu/~jsellers/garvansellerssmoot2024.pdf}{https://www.d.umn.edu/$\sim$jsellers/garvansellerssmoot2024.pdf}.

\section{Frobenius Partition Functions}\label{sectdef}

As noted above, the general study of Frobenius partition functions was initiated by Andrews in 1984 \cite{Andrews}.  This is now a broad area of study, and we will concentrate on a recent synthesis by Jiang, Rolen, and Woodbury \cite{Jiang} under which the binary nature of the Andrews--Sellers congruences can be recognized.

\begin{definition}{(Jiang--Rolen--Woodbury)}\label{jrwd}
Let $n,k\in\mathbb{Z}_{\ge 1}$ and $\beta\in\mathbb{Z}+\frac{k}{2}$ nonnegative.  A $(k,\beta)$\textit{-colored generalized Frobenius partition} of $n$ is an array of the form
\begin{align*}
\begin{pmatrix}
  a_1 & a_2 & ... & a_r \\
  b_1 & b_2 & ... & b_s 
 \end{pmatrix}
\end{align*} satisfying the following:
\begin{itemize}
\item Each $a_i,b_j$ belongs to one of $k$ copies of $\mathbb{Z}_{\ge 0}$,
\item Each row is decreasing with respect to lexicographic ordering,
\item $r+s\neq 0$
\item $r-s = \beta-\frac{k}{2}$
\item $n = r+\sum_{0\le i\le r}a_i + \sum_{0\le j\le s}b_j.$
\end{itemize}  The number of such arrays is denoted $c\psi_{k,\beta}(n)$, and the associated generating function over the variable $q$ is denoted $\mathrm{C}\Psi_{k,\beta}(q)$.
\end{definition}

By and large, the generating functions for $\mathrm{C}\Psi_{k,\beta}(q)$ are not \textit{a priori} obvious; indeed, a significant portion of \cite{Jiang} is devoted specifically to the problem of constructing such generating functions.  In the cases $k=2$, $\beta=0,1$, the generating functions are simple to state \cite{Andrews,Drake,Jiang}:
\begin{align*}
\mathrm{C}\Psi_{2,1}(q) := \sum_{n=0}^{\infty}c\phi_2(n)q^n &= \sum_{n=0}^{\infty}c\psi_{2,1}(n)q^n = \frac{(q^2;q^2)_{\infty}^5}{(q;q)_{\infty}^4(q^4;q^4)_{\infty}^2},\\
\mathrm{C}\Psi_{2,0}(q) := \sum_{n=0}^{\infty}c\psi_2(n)q^n &= \sum_{n=0}^{\infty}c\psi_{2,0}(n)q^n = \frac{(q^4;q^4)_{\infty}^2}{(q;q)_{\infty}^2(q^2;q^2)_{\infty}}.
\end{align*}  Here we use the standard $q$-Pochhammer symbol
\begin{align*}
(q^a;q^b)_{\infty} := \prod_{m=0}^{\infty}\left( 1 - q^{bm+a} \right).
\end{align*}  The close relationship between these two Frobenius functions is not clear from the generating functions themselves.  On the other hand, Definition \ref{jrwd} reveals that the two are combinatorically close.  The key difference is that for $c\psi_{2,1}(n)$ we must count arrays in which $r=s$; for $c\psi_{2,0}(n)$, we must have $s=r+1$.

\section{Proof of Theorem \ref{Thm12}}\label{mainthmetc}

A common approach to proving a congruence family corresponding to the coefficients of a given modular form is to construct a sequence $\mathcal{L}:=(L_{\alpha})_{\alpha\ge 1}$ of functions on an associated manifold---which for our purposes will always be the classical modular curve $\mathrm{X}_0(N)$ for some positive integer $N$.  Indeed, except for Section \ref{sectaddit}, we will almost always work with $\mathrm{X}_0(20)$.  Each $L_{\alpha}$ enumerates a given case of the congruence family, and each function is related to its predecessor via a certain linear operator.  The difficulty then is to express each $L_{\alpha}$ in terms of some convenient reference functions.

We will define a pair of function sequences,

\begin{align}
\mathcal{L}^{(\beta)} := \left( L_{\alpha}^{(\beta)} \right)_{\alpha\ge 1}
\end{align} with $\beta\in\{0,1\}$, by

\begin{align}
L_{\alpha}^{(\beta)} := \frac{1}{\mathrm{C}\Psi_{2,\beta}\left(q^{1+4a}\right)}\cdot \sum_{3\cdot 2^{\beta+1}n\equiv (-1)^{\beta+1}\bmod{5^{\alpha}}} c\psi_{2,\beta}(n) q^{\left\lfloor n/5^{\alpha}\right\rfloor + \beta},
\end{align} where $a\in\{0,1\}$, $\alpha\equiv a\pmod{2}$.

To relate each $L_{\alpha}^{(\beta)}$ to its successor, we construct a set of linear operators $U^{(\beta)}_{a}$.  In particular, for $\beta=0,1$, we have
\begin{align}
U^{(\beta)}_{1}\left( \sum_{n\ge N}a(n)q^n \right) = \sum_{5n\ge N}a(5n)q^n
\end{align} is the standard $U_5$ operator (see \cite[Chapter 8]{Knopp} for a list of the standard properties of $U_5$), and
\begin{align}
U^{(\beta)}_{0}\left( f \right) = U_{5}\left( \mathcal{A}^{(\beta)}\cdot f \right),
\end{align} where
\begin{align}
\mathcal{A}^{(\beta)} = q^{2(3\beta-2)}\frac{\mathrm{C}\Psi_{2,\beta}(q)}{\mathrm{C}\Psi_{2,\beta}(q^{25})}.
\end{align}  It is a comparatively standard procedure in the theory to show that for all $\alpha\ge 1$, with $a\in\{0,1\}$ matching the parity of $\alpha$, we have
\begin{align}
U^{(\beta)}_{a}\left( L_{\alpha}^{(\beta)} \right) = L_{\alpha+1}^{(\beta)}.
\end{align}  See, for example, \cite[p. 23]{Atkin}, or \cite[Lemma 2.13]{Smoot}.

If we can show that $L_{\alpha}^{(\beta)}$ is divisible by $5^{\alpha}$, then we have proved Theorems \ref{Thm1} and \ref{Thm12}.  If we make the substitution $q=e^{2\pi i\tau}$ with $\tau\in\mathbb{H}$, then these function sequences are modular functions over the congruence subgroup $\Gamma_0(20)$; equivalently, they correspond to functions on the modular curve $\mathrm{X}_0(20)$.

This modular curve is a compact Riemann surface of cusp count 6 and genus 1.  If we were to consider, say, the space of functions on $\mathrm{X}_0(20)$ which live at the cusp $[0]$ (that is, the meromorphic functions which are holomorphic along the entire curve except the cusp $[0]$), and denote this space as $\mathcal{M}^0\left( \mathrm{X}_0(20) \right)$, then the Weierstra\ss\ gap theorem (e.g., \cite{Paule2}) predicts that this space is a rank 2 $\mathbb{C}[x]$-module for some function $x$---that is, there exist two functions, which we will here denote $x$ and $y$, such that 
\begin{align}
\mathcal{M}^0\left( \mathrm{X}_0(20) \right) = \mathbb{Z}[x]+y\mathbb{Z}[x].
\end{align}  In this case, one may show with some ease that a sufficient choice of $x$, $y$ is
\begin{align}
x &= q\frac{(q^2;q^2)_{\infty}(q^{10};q^{10})_{\infty}^3}{(q;q)_{\infty}^3(q^5;q^5)_{\infty}},\\
y &= q^2\frac{(q^2;q^2)_{\infty}^2(q^4;q^4)_{\infty}(q^5;q^5)_{\infty}(q^{20};q^{20})_{\infty}^3}{(q;q)_{\infty}^5(q^{10};q^{10})_{\infty}^2}.
\end{align}  Because the functions $L_{\alpha}^{(\beta)}$ are modular, they are equivalently meromorphic functions on $\mathrm{X}_0(20)$ with possible poles only at the cusps.

One of the features which substantially simplifies the proofs of other congruence families is that when the cusp count is 2, the associated functions $L_{\alpha}$ can be easily made to have positive order at one cusp.  This forces them to have a pole at the other cusp, and consequently to be representable as a polynomial in the generators of the functions which live at that cusp.  This is true, for example, in the cases of Ramanujan's classic congruence families modulo powers of 5, 7, 11 (resp.) for the unrestricted partition function $p(n)$, in which the associated modular curves are $\mathrm{X}_0(\ell)$ with $\ell$ equal to 5, 7, 11 (resp.) \cite{Watson,Atkin}.

Notice that this is not possible in our current case.  The cusp count is 6, and there is no guarantee that $L_{\alpha}^{(\beta)}$ live at the cusp $[0]$ (and indeed, they do not).  One possible resolution is to multiply each $L_{\alpha}^{(\beta)}$ by an eta quotient which lives at $[0]$ and which has positive order high enough to eliminate all of the poles of $L_{\alpha}^{(\beta)}$ except for one at $[0]$.  For us, such a function arises as powers of $1+5x$.  This places $L_{\alpha}^{(\beta)}$ in a localized ring of modular functions.  For example,

\begin{align}
L_{1}^{(1)} = \frac{5}{(1+5x)^2}&\bigg( 4 x + 137 x^2 + 1704 x^3 + 10080 x^4 + 28800 x^5 + 32000 x^6\notag\\ &-y(20 + 400 x + 3040 x^2 + 10240 x^3 + 12800 x^4)\bigg).
\end{align}  We recognize divisibility of the function by 5 in this identity.  Similar identities can easily be computed for larger $\alpha$, and it can be shown that
\begin{align}
L_{\alpha}^{(1)}\in\mathbb{Z}[x]_{\mathcal{S}}+y\mathbb{Z}[x]_{\mathcal{S}},
\end{align} for
\begin{align}
\mathcal{S} := \left\{ (1+5x)^n : n\ge 0 \right\}.
\end{align}  The possibility of actually proving the whole family of congruences by manipulation of these rational polynomials is naturally very alluring.  Such a class of proofs does exist, and has been exhibited for other difficult congruence families in previous papers, e.g., \cite{Smoot}, \cite{Smoot0}, \cite{Baner7}.

This technique fails spectacularly in the cases of the families in Theorems \ref{Thm1} and \ref{Thm12}.  It is certainly true that manipulation of these rational polynomial identities allows one to confirm as many cases of either theorem as we like.  However, it is not currently possible to complete the induction along these lines owing to our current inability to properly stabilize the associated congruence ideal sequence (see \cite[Section 6]{Baner7} for an explicit description of this approach).

Paule and Radu concentrated on a different approach in \cite{Paule}.  They tried to express $L_{\alpha}^{(\beta)}$ almost entirely in terms of functions on a modular curve of lower level.  Let us consider
\begin{align}
t =&\ \frac{\eta(5\tau)^6}{\eta(\tau)^6} = q\left(\frac{(q^5;q^5)_{\infty}}{(q;q)_{\infty}}\right)^6.
\end{align}  This is the Hauptmodul for $\mathrm{X}_0(5)$, a curve of genus 0 and cusp count 2.  Indeed, $t$ is the critical function for proving Ramanujan's congruence family for $p(n)$ modulo powers of 5 (see \cite{Watson} or \cite[Chapters 7-8]{Knopp}).  Being modular over $\Gamma_0(5)$, it is clearly modular over $\Gamma_0(20)$.  It does not live at the cusp $[0]$, but we can perform the same trick we did with $L_1^{(1)}$---we can represent $t$ as a rational polynomial.  Indeed, we have
\begin{align}
t = \frac{x(1+4x)^2}{1+5x}.\label{trep}
\end{align}  Notice that if we multiply through by $1+5x$, substitute $z=1+5x$, and clear fractions, we then have
\begin{align}
(1+5x)t &= x(1+4x)^2\\
zt &= \frac{-1}{125}(1+7z+8z^2-16z^3),\\
-125zt &= 1+7z+8z^2-16z^3.
\end{align}  If we then isolate 1, divide by $z$, and shift back to $x$, we of course have
\begin{align}
1 &= -125tz - 7z - 8z^2 + 16z^3,\\
\frac{1}{z} &= -125t - 7 - 8z + 16z^2,\\
\frac{1}{1+5x} &= 1 - 125 t + 120 x + 400x^2.
\end{align}  Thus through an elementary set of manipulations, we can now represent elements of our localization ring as polynomials in $t$, $x$, $y$.  Indeed, one sees that $x^3$ can be written as a polynomial in $t$, $x$, $x^2$.  As such, we can express each $L_{\alpha}^{(\beta)}$ as a polynomial in $t,x,y$, but with the powers of $x$ restricted to 2 or lower (and the powers of $y$ restricted to 1 or lower).  Thus, our representations will be ``mostly" in terms of the function $t$.

Nevertheless, from these manipulations, we would guess that the associated polynomials would occur as members of a rank 6 $\mathbb{Z}[t]$-module (corresponding to the basis elements 1, $x$, $x^2$, $y$, $yx$, $yx^2$).  The second (and most important) breakthrough of Paule and Radu was to recognize that each $L_{\alpha}^{(1)}$ is expressible in terms of a $\mathbb{Z}[t]$-module of rank 3.  Indeed, each $L_{\alpha}^{(1)}$ can be placed in one of a pair of rank 2 $\mathbb{Z}[t]$-modules depending on the parity of $\alpha$; that is,
\begin{align}
L_{\alpha}^{(1)}\in\mathbb{Z}[t]+p^{(1)}_{a}\mathbb{Z}[t]
\end{align} for a pair of functions $p^{(1)}_{0}$, $p^{(1)}_{1}$.  It is the set of coefficients in this representation which may be closely studied to prove divisibility by ever increasing powers of 5.

For instance,
\begin{align}
L_{1}^{(1)} = -5t+25p^{(1)}_{1}.\label{L11rep}
\end{align}  The question of how to derive this reduced rank representation is a serious one; for the time being we postpone discussion of this to Section \ref{sectinvol}, and instead simply give the functions in terms of $1, x, x^2, y, xy, x^2y$, and powers of $t$:
\begin{align}
p^{(1)}_{0} =&\ 85 t + 625 t^2 - 24 x - 64 x^2 - 4 y - 1000 t y - 12500 t^2 y + 992 x y + 12000 t x y + 3520 x^2 y + 40000 t x^2 y,\label{be1a0}\\
p^{(1)}_{1} =&\ t + 25 t^2 - 8 t x - 4 y - 8 t y - 500 t^2 y - 32 x y + 480 t x y - 64 x^2 y + 1600 t x^2 y.\label{be1a1}
\end{align}

Upon discovering the congruence family of Theorem \ref{Thm12}, we guessed that the associated functions $L_{\alpha}^{(0)}$ could be represented in a similar manner.  We eventually discovered that 
\begin{align}
L_{\alpha}^{(0)}\in\mathbb{Z}[t]+p^{(0)}_{a}\mathbb{Z}[t],
\end{align} with
\begin{align}
p^{(0)}_{0} =&\ 1 + 5 t + 64 x + 600 t x + 208 x^2 + 2000 t x^2 + 4 y + 1000 t y + 12500 t^2 y - 992 x y - 12000 t x y\notag\\ &- 3520 x^2 y - 40000 t x^2 y,\label{be0a0}\\
p^{(0)}_{1} =&\ 1 + 9 t + 8 x + 64 t x + 16 x^2 + 80 t x^2 + 4 y + 8 t y + 500 t^2 y + 32 x y - 480 t x y + 64 x^2 y - 1600 t x^2 y.\label{be0a1}
\end{align}  This seemed reasonable enough; however, comparison of the actual representations of $L_{\alpha}^{(0)}$, $L_{\alpha}^{(1)}$ led us to a much more stunning result.  Let us give $L_{1}^{(0)}$:
\begin{align}
L_{1}^{(0)} = -5t+25p^{(0)}_{1}.\label{L01rep}
\end{align} Comparing (\ref{L11rep}) and (\ref{L01rep}), one finds that the two identities formally match.  Indeed, we have verified that \textit{all} of the formal representations for $L_{\alpha}^{(0)}$ and $L_{\alpha}^{(1)}$ match exactly---they differ only in the associated module generators $p^{(\beta)}_{\alpha}$.

\begin{theorem}\label{fieldfuncfix}
Define 
\begin{align}
\mathcal{R}^{(1)}:&=\mathbb{C}(t)+p^{(1)}_{0}\mathbb{C}(t)\ + p^{(1)}_{1}\mathbb{C}(t)\\
\mathcal{R}^{(0)}:&=\mathbb{C}(t)+p^{(0)}_{0}\mathbb{C}(t)\ + p^{(0)}_{1}\mathbb{C}(t)
\end{align} and the ring mappings
\begin{align}
\sigma^{(\beta)}&: \mathcal{R}^{(\beta)} \longrightarrow\mathcal{R}^{(1-\beta)}\\
&:\left\{\begin{matrix}
& t\longmapsto t,\\
& p^{(\beta)}_{a}\longmapsto p^{(1-\beta)}_{a}.
\end{matrix}\right.
\end{align}  Then
\begin{align}
\sigma^{(\beta)}\left( L^{(\beta)}_{\alpha} \right) = L^{(1-\beta)}_{\alpha}.
\end{align}
\end{theorem}

\subsection{Proof of Theorem \ref{fieldfuncfix}}\label{proofAtkin1}

One straightforward, if tedious, proof of Theorem \ref{fieldfuncfix} involves directly computing and comparing $U^{(\beta)}_a \left(t^m\right)$, $U^{(\beta)}_a \left(t^m p^{(0)}_{a}\right)$ for $\beta=0,1$, $a=0,1$, and $m$ ranging over 5 consecutive values.  One can then invoke a modular equation for $t$, e.g., the equation used by Paule and Radu in \cite{Paule}.

However, the existence of such a ring mapping system as $\sigma^{(\beta)}$ suggests that there exists some explicit means of transforming the functions corresponding to $c\psi_{2,0}(n)$ into those corresponding to $c\psi_{2,1}(n)$.  Our first instinct on the matter was to attempt a connection between $\mathrm{C}\Psi_{2,\beta}$, $\mathrm{C}\Psi_{2,1-\beta}$ via an Atkin--Lehner involution.

At first no direct connection could be found.  On closer examination, the first author realized that we can indeed create an explicit mapping between these generating functions by first mapping $q\mapsto -q$, then applying an Atkin--Lehner involution, and finally mapping $q\mapsto -q$ a second time.  Equivalently, we conjugate a certain Atkin--Lehner involution by a half-integer valued matrix which equates to sending $q\mapsto -q$.

More specifically, we consider the operator
\begin{align*}
W := \begin{pmatrix}
4 & -1\\
100 & -24
\end{pmatrix}.
\end{align*} Let us start by taking the functions $\mathcal{A}^{(\beta)}$ which are necessary for construction of our $U^{(\beta)}_{\alpha}$ operators.  If we take, say, $\beta=1$, and we first map $q\mapsto -q$, then we have
\begin{align}
\mathcal{A}^{(1)}\longmapsto \frac{\eta(\tau)^4 \eta(4\tau)^2 \eta(50\tau)^7}{\eta(2\tau)^7 \eta(25\tau)^4 \eta(100\tau)^2}.
\end{align}  We want to send $\tau$ to $W\tau$.  See Section \ref{sectinvol} below for a detailed treatment of how $W$ affects each eta factor, and indeed how it affects our relevant spaces of modular functions.  This gives us
\begin{align}
\frac{\eta(\tau)^4 \eta(4\tau)^2 \eta(50\tau)^7}{\eta(2\tau)^7 \eta(25\tau)^4 \eta(100\tau)^2} \longmapsto \frac{\eta(\tau)^2 \eta(4\tau)^4 \eta(50\tau)^7}{\eta(2\tau)^7 \eta(25\tau)^2 \eta(100\tau)^4}.
\end{align}  Finally, applying $q\mapsto -q$, we achieve
\begin{align}
\frac{\eta(\tau)^2 \eta(4\tau)^4 \eta(50\tau)^7}{\eta(2\tau)^7 \eta(25\tau)^2 \eta(100\tau)^4}\longmapsto \mathcal{A}^{(0)}.
\end{align}  We now have a direct means of connecting the two.

If we examine the composition of these mappings, we find that the matrix representing the overall action is
\begin{align}
\begin{pmatrix}
1 & 1/2\\
0 & 1
\end{pmatrix}W\begin{pmatrix}
1 & 1/2\\
0 & 1
\end{pmatrix} = \begin{pmatrix}
54 & 14\\
100 & 26
\end{pmatrix},\label{swsA}
\end{align} which will reduce via the associated transformation to
\begin{align*}
\gamma =\begin{pmatrix}
27 & 7\\
50 & 13
\end{pmatrix}\in\Gamma_0(10).
\end{align*}  Notice that by this transformation, we have 
\begin{align*}
t\left(\gamma\tau\right) = t(\tau),\\
x\left(\gamma\tau\right) = x(\tau),
\end{align*} as the former is modular over $\Gamma_0(5)$, and the latter is modular over $\Gamma_0(10)$.  We need only to understand how $\gamma$ affects $y$.  Applying the transformation and accounting for the Nebentypus factor, we can show that
\begin{align*}
y\left(\gamma\tau\right) = -\frac{1}{4}(1+6x+4y).
\end{align*}  Knowing that $t$ and $x$ are invariant, we can make this substitution for $y$ into $p_1^{(1)}$, one gets
\begin{align*}
p^{(1)}_{1}\left(\gamma\tau\right) =&\ t + 25 t^2 - 8 t x - y\left(\gamma\tau\right)\left(  4 + 8 t + 500 t^2 + 32 x  - 480 t x  + 64 x^2  - 1600 t x^2  \right)\\
=&\ 1 + 9 t + 8 x + 64 t x + 16 x^2 + 80 t x^2 + 4 y + 8 t y + 500 t^2 y + 32 x y - 480 t x y + 64 x^2 y - 1600 t x^2 y\\ 
=&\ p^{(0)}_{1}.
\end{align*}  Similarly, we have
\begin{align*}
p^{(1)}_{0}\left(\gamma\tau\right) =&\ p^{(0)}_{0}.
\end{align*}  Thus $\sigma^{(1)}$ sends $f(\tau)$ to $f(\gamma\tau)$, and we can define
\begin{align*}
\sigma^{(1)}\left(f\right) = f\left(\gamma\tau\right).
\end{align*}  We have
\begin{align}
L_1^{(0)} &= U_5\left( \mathcal{A}^{(0)} \right) = U_5\left( \mathcal{A}^{(1)}(\gamma\tau) \right) = L_1^{(1)}(\gamma\tau) = \sigma^{(1)}\left(L_{1}^{(1)}\right).
\end{align}
Supposing that for some $\alpha\ge 2$ we have $L_{\alpha-1}^{(0)} = L_{\alpha-1}^{(1)}(\gamma\tau)$.  Again allowing $\alpha \equiv a \pmod{2}$ for $a\in \{0,1\}$,
\begin{align}
L_{\alpha}^{(0)} &= U_{1-a}^{(0)}\left( L_{\alpha-1}^{(0)} \right)\\
&= U_{5}\left( \left(\mathcal{A}^{(0)}\right)^{a} L_{\alpha-1}^{(0)} \right)\\
&= U_{5}\left( \left(\mathcal{A}^{(1)}(\gamma\tau)\right)^{a} L_{\alpha-1}^{(1)}(\gamma\tau) \right)\label{switchA}\\
&= U_{5}\left( \left(\mathcal{A}^{(1)}\right)^{a} L_{\alpha-1}^{(1)} \right)(\gamma\tau)\label{switchB}\\
&= L_{\alpha}^{(1)}(\gamma\tau)\\
&= \sigma^{(1)}\left(L_{\alpha}^{(1)}\right).
\end{align}  Ordinarily, we could not justify (\ref{switchA})-(\ref{switchB}).  However, in our case the action of $\gamma$ is equivalent to the composition of transformations in (\ref{swsA}), and the factors $W$, $\left(\begin{smallmatrix}
1 & 1/2\\
0 & 1
\end{smallmatrix}\right)$ commute with the $U_5$ operator.

Thus induction on $\alpha$ gives us
\begin{align*}
\sigma^{(1)}\left(L_{\alpha}^{(1)}\right) = L_{\alpha}^{(1)}\left(\gamma\tau\right) = L_{\alpha}^{(0)}
\end{align*} for all $\alpha\ge 1$, whereupon we have Theorem \ref{fieldfuncfix}.

Of course, Theorem \ref{Thm1} was proved by showing that $L^{(1)}_{\alpha}$ is divisible by $5^{\alpha}$.  Therefore, $\sigma^{(1)}\left( L^{(1)}_{\alpha} \right) = L^{(0)}_{\alpha}$ will also be so divisible.  With that we immediately have Theorem \ref{Thm12} and Corollary \ref{Thm13}.

Notice that we can similarly show that
\begin{align*}
\sigma^{(0)}\left(L_{\alpha}^{(0)}\right) = L_{\alpha}^{(0)}\left(\gamma^{-1}\tau\right) = L_{\alpha}^{(1)},
\end{align*} with 
\begin{align*}
\gamma =\begin{pmatrix}
13 & -7\\
-50 & 27
\end{pmatrix}\in\Gamma_0(10).
\end{align*}  Thus, Theorems \ref{Thm1} and \ref{Thm12} are equivalent.

\section{Analogous Examples in the Literature}\label{sectaddit}

Notice that $\sigma^{(\beta)}$ is a mapping on the ring of modular functions over $\Gamma_0(20)$ which fixes the functions over $\Gamma_0(5)$ (since $t$ is a Hauptmodul for functions over the latter).  Thus, the two congruence families are closely related to functions on modular function fields.  This very likely poses significant implications for the question of how congruence families relate through mappings of modular functions on various congruence subgroups.  The fact that the explicit nature of the mapping is simply a transformation within the intermediate group $\Gamma_0(10)$ suggests that such relationships may be quite commonplace with congruence families, particularly when the associated modular curve has composite level.

We have come to recognize that the treatment given above may be applied in an analogous manner to a large number of other recent results in the literature.  Below we give three different examples.  One involves partition functions ($\phi_2(n)$ and $\psi_2(n)$) in which infinite congruence families do not appear to exist, but for which individual congruences have been found.  Another involves congruence families related to the crank parity function and two other partition functions.  A third relates to a pair of infinite internal congruence families for functions closely related to Ramanujan's theta functions.

The most important common property of all of these cases is that the level of the corresponding modular curve is always composite.  The sheer variety of different problems for which our approach is applicable is especially striking.

\subsection{Congruences Related to Other Frobenius Partitions}\label{subsecotherfrob}

In the first place, we have discovered a similar relationship appears to apply to other Frobenius partitions.  Consider the functions $\phi_2(n)$ and $\psi_2(n)$ which, for want of space, we will merely define in terms of their generating functions \cite{Andrews, Drake}:
\begin{align*}
\Phi_{2}(q) := \sum_{n=0}^{\infty}\phi_2(n)q^n &= \frac{(q^4;q^4)_{\infty} (q^{6};q^{6})_{\infty}^2}{(q;q)_{\infty}(q^2;q^2)_{\infty}(q^3;q^3)_{\infty}(q^{12};q^{12})_{\infty}},
\end{align*}
\begin{align*}
\Psi_{2}(q) := \sum_{n=0}^{\infty}\psi_2(n)q^n &= \frac{(q^2;q^2)_{\infty}^2 (q^{12};q^{12})_{\infty}}{(q;q)_{\infty}^2(q^4;q^4)_{\infty}(q^6;q^6)_{\infty}}.
\end{align*}  Given that our main result interrelates the divisibility properties of $c\phi_2(n)$ with those of $c\psi_2(n)$, it is very tempting to look for a similar relationship between $\phi_2(n)$ and $\psi_2(n)$.  Neither of these functions appear to exhibit infinite families of congruences.  However, Andrews did prove \cite[Corollary 10.1]{Andrews} the following:
\begin{theorem}\label{thmphi2}
For all $n\ge 0$, $\phi_2(5n+3)\equiv 0\pmod{5}$.
\end{theorem}  More recently, Eckland and the second author proved \cite{Eckland} a similar result for $\psi_2(n)$:
\begin{theorem}\label{thmpsi2}
For all $n\ge 0$, $\psi_2(5n+4)\equiv 0\pmod{5}$.
\end{theorem}  These two properties turn out to be equivalent to one another in the same manner of the congruence families for Theorems \ref{Thm1} and \ref{Thm12}.  Indeed, we can go further.  Theorem \ref{thmphi2} is consequent to the following witness identity derived by the first author:

\begin{align}
\frac{1}{\Phi_2(q^5)}\sum_{n=0}^{\infty}\phi_2(5n+3)q^{n+1} = 5t\cdot (-1+5P_1^{(1)}+60P_2^{(1)}-3P_3^{(1)})\in\mathcal{M}\left(\mathrm{X}_0(60)\right)\label{phi2id},
\end{align} with $t$ defined as before, and

\begin{align*}
P_1^{(1)} &=\frac{\eta(4\tau)\eta(5\tau)^{6}\eta(6\tau)^{2}\eta(60\tau)}{\eta(\tau)\eta(2\tau)\eta(3\tau)^{2}\eta(10\tau)^{2}\eta(12\tau)\eta(15\tau)\eta(20\tau)\eta(30\tau)},\\
P_2^{(1)} &=\frac{\eta(4\tau)\eta(5\tau)\eta(6\tau)^{2}\eta(30\tau)^{4}\eta(60\tau)}{\eta(\tau)\eta(2\tau)\eta(3\tau)^{2}\eta(10\tau)\eta(12\tau)\eta(15\tau)^{2}\eta(20\tau)},\\
P_3^{(1)} &=\frac{\eta(3\tau)\eta(4\tau)\eta(6\tau)^{2}\eta(10\tau)\eta(15\tau)\eta(60\tau)}{\eta(2\tau)\eta(5\tau)^{2}\eta(12\tau)\eta(20\tau)\eta(30\tau)^{2}}.
\end{align*}  Let us refer to the left-hand-side of (\ref{phi2id}) as $L^{(1)}$.

This identity can be derived by a straightforward but tedious process involving the modular cusp analysis, and we walk through this process in our Mathematica supplement (See the end of Section 1).  However, were Theorems \ref{thmphi2} and \ref{thmpsi2} equivalent in the manner of Theorems \ref{Thm1} and \ref{Thm12}, we would expect an analogous identity to hold for the generating function for $\psi_2(n)$.  Indeed, let us define

\begin{align*}
L^{(0)} := \frac{1}{\Psi_2(q^5)}\sum_{n=0}^{\infty}\psi_2(5n+4)q^{n+1}\in\mathcal{M}\left(\mathrm{X}_0(60)\right),
\end{align*} in a notation befitting for our approach.  We want to derive an identity for $L^{(0)}$ using $L^{(1)}$.  Consider the Atkin--Lehner involution
\begin{align*}
W =\begin{pmatrix}
4 & 1\\
300 & 76
\end{pmatrix}
\end{align*}  Conjugating once again with our sign-change matrix, we have
\begin{align*}
\begin{pmatrix}
1 & 1/2\\
0 & 1
\end{pmatrix}W\begin{pmatrix}
1 & 1/2\\
0 & 1
\end{pmatrix} = \begin{pmatrix}
154 & 116\\
300 & 226
\end{pmatrix}.
\end{align*}  We can immediately see that the associated action is equivalent to
\begin{align*}
\gamma =\begin{pmatrix}
77 & 58\\
150 & 113
\end{pmatrix}\in\Gamma_0(5).
\end{align*}  Thus $t$ is invariant under this transformation, and it can be shown that
\begin{align*}
t(\gamma\tau) &= t(\tau),\\
P_1^{(1)}(\gamma\tau) &=\frac{\eta(2\tau)^{2}\eta(5\tau)^{7}\eta(12\tau)\eta(20\tau)\eta(30\tau)^{2}}{\eta(\tau)^{2}\eta(3\tau)\eta(4\tau)\eta(6\tau)\eta(10\tau)^{5}\eta(15\tau)^{2}\eta(60\tau)},\\
P_2^{(1)}(\gamma\tau) &=\frac{\eta(2\tau)^{2}\eta(5\tau)^{2}\eta(12\tau)\eta(20\tau)\eta(30\tau)^{7}}{\eta(\tau)^{2}\eta(3\tau)\eta(4\tau)\eta(6\tau)\eta(10\tau)^{4}\eta(15\tau)^{3}\eta(60\tau)^{2}},\\
P_3^{(1)}(\gamma\tau) &=\frac{\eta(2\tau)^{2}\eta(3\tau)^{2}\eta(12\tau)\eta(20\tau)\eta(30\tau)}{\eta(\tau)\eta(4\tau)\eta(5\tau)\eta(6\tau)\eta(10\tau)^{2}\eta(60\tau)}.
\end{align*}  So we have a means for naturally transforming the right-hand side of (\ref{phi2id}) to something reasonable.  To show that this is in fact $L^{(0)}$, we note that

\begin{align*}
\mathcal{B}^{(1)} = q^2\frac{\Phi_{2}(q)}{\Phi_{2}(q^{25})} = q^2\frac{(q^4;q^4)_{\infty} (q^{6};q^{6})_{\infty}^2(q^{25};q^{25})_{\infty}(q^{50};q^{50})_{\infty}(q^{75};q^{75})_{\infty}(q^{300};q^{300})_{\infty}}{(q;q)_{\infty}(q^2;q^2)_{\infty}(q^3;q^3)_{\infty}(q^{12};q^{12})_{\infty}(q^{100};q^{100})_{\infty} (q^{150};q^{150})_{\infty}^2}
\end{align*} is a modular function over $\Gamma_0(300)$.  The same is true for
\begin{align*}
\mathcal{B}^{(0)} =\frac{1}{q^4}\frac{\Psi_{2}(q)}{\Psi_{2}(q^{25})} = \frac{1}{q^4}\frac{(q^2;q^2)_{\infty}^2 (q^{12};q^{12})_{\infty}(q^{25};q^{25})_{\infty}^2(q^{100};q^{100})_{\infty}(q^{150};q^{150})_{\infty}}{(q;q)_{\infty}^2(q^4;q^4)_{\infty}(q^6;q^6)_{\infty}(q^{50};q^{50})_{\infty}^2 (q^{300};q^{300})_{\infty}}
\end{align*}  Applying the $U_5$ operator to each of these functions will yield an object modular over $\Gamma_0(60)$.  Indeed,
\begin{align*}
U_5\left(\mathcal{B}^{(1)}\right) &= \frac{1}{\Phi_{2}(q^{5})}U_5\left(q^2\sum_{n=0}^{\infty}\phi_2(n)q^n\right)\\
&= \frac{1}{\Phi_{2}(q^{5})}U_5\left(\sum_{n\ge 2}\phi_2(n-2)q^n\right)\\
&= \frac{1}{\Phi_{2}(q^{5})}\sum_{5n\ge 2}\phi_2(5n-2)q^n\\
&= \frac{1}{\Phi_{2}(q^{5})}\sum_{n\ge 0}\phi_2(5n+3)q^{n+1}\\
&= L^{(1)}.
\end{align*}  By a similar manner, it may be shown that
\begin{align*}
U_5\left(\mathcal{B}^{(0)}\right) &= \frac{1}{\Psi_{2}(q^{5})}\sum_{n\ge 0}\psi_2(5n+4)q^{n}\\
&= L^{(0)}.
\end{align*}  If we examine the effect of $\gamma$ on $\mathcal{B}^{(\beta)}$, a manipulaiton of the corresponding eta quotients reveals that
\begin{align*}
\mathcal{B}^{(1)}(\gamma\tau) = \mathcal{B}^{(0)},
\end{align*} and
\begin{align*}
L^{(1)}(\gamma\tau) = U_5\left(\mathcal{B}^{(1)}(\gamma\tau)\right) &= U_5\left(\mathcal{B}^{(0)}\right) = L^{(0)}.
\end{align*}  Thus,
\begin{align}
\frac{1}{\Psi_2(q^5)}\sum_{n=0}^{\infty}\psi_2(5n+4)q^{n} = 5t\cdot (-1+5P_1^{(0)}+60P_2^{(0)}-3P_3^{(0)}),
\end{align} and we have derived the corresponding witness identity for Theorem \ref{thmpsi2}.  Notice that again, the coefficients in the witness identity---including the factors relevant to the divisibility property in question---are identical for both $L^{(1)}$ and $L^{(0)}$.

It is noteworthy that while individual divisibility properties for $\phi_2(n)$ and $\psi_2(n)$ exist, they tend to be rare, and no infinite families divisible by, say, powers of 5 are known for either.  But this lack of congruence families is in itself interesting.  The existence of the equivalence between congruences mod 5 for these two functions also implies the equivalence of a lack of infinite families mod powers of 5 for either; if one such family existed for $\phi_2(n)$, we would expect to find an equivalent family exhibited by $\psi_2(n)$, and vice versa.

Thus the equivalences that we are constructing between various partition functions imply equivalences of nonexistence, as well as existence, of divisibility properties of interest.

\subsection{Congruences Related to the Crank Statistic}

A second tantalizing example is extremely recent and due to a collaboration between the first author and Connor Morrow \cite{GarvanM}.  In this section we will use $\alpha$ and $\beta$ as arithmetic functions, in contrast to the rest of this paper wherein both are used as indices.

We consider a congruence family for the function $\alpha(n)$, which we will here define simply through its generating function
\begin{align*}
\sum_{n=0}^{\infty} \alpha(n)q^n := \frac{(q^2;q^2)_{\infty}^3}{(q;q)_{\infty}^2}.
\end{align*}  This function bears close enumeration relationships to multiple other functions, including the restricted partition function $\overline{\mathcal{EO}}(n)$ studied by Andrews \cite{Andrews2}, and the 1-shell TSPP function studied by Hirschhorn with the second author \cite{Hirschhorn}, and Chern \cite{Chern2}.

\begin{theorem}\label{ThmMG1}
Let $n,m\in\mathbb{Z}_{\ge 1}$ such that $6n\equiv -1\pmod{5^{m}}$.  Then $\alpha(n)\equiv 0\pmod{5^{\left\lceil m/2\right\rceil}}$.
\end{theorem}  Notice the similarity of the internal constraint to Theorem \ref{Thm12}.  The functions $L_{\alpha}$ associated with this congruence family are modular over $\Gamma_0(10)$.

Garvan and Morrow soon realized that this congruence family has a strange association with the crank parity function $\beta(n) := M_e(n) - M_o(n)$.  Here $M_e$ ($M_o$) counts the number of partitions of $n$ with even (odd) crank, respectively.  The generating function for $\beta(n)$ is
\begin{align*}
\sum_{n=0}^{\infty} \beta(n)q^n = \frac{(q;q)_{\infty}^3}{(q^2;q^2)_{\infty}^2}.
\end{align*}  The crank parity function exhibits a congruence family that was proved by Choi, Kang, and Lovejoy in 2009 \cite{Choi}:

\begin{theorem}\label{ThmMG2}
Let $n,m\in\mathbb{Z}_{\ge 1}$ such that $24n\equiv 1\pmod{5^{2m+1}}$.  Then $\beta(n)\equiv 0\pmod{5^{m+1}}$.
\end{theorem}

Let $L_m^{(\alpha)}$ be the corresponding modular function which enumerates the $m$-th case of Theorem \ref{ThmMG1}, and let $L_m^{(\beta)}$ be similarly defined for Theorem \ref{ThmMG2}.  For example,
\begin{align*}
L_1^{(\alpha)} &= \frac{(q^5;q^5)_{\infty}^2}{(q^{10};q^{10})_{\infty}^3}\sum_{n=0}^{\infty} \alpha(5n+4)q^n = 5\frac{(q^{2};q^{2})_{\infty}^2(q^{5};q^{5})_{\infty}^4}{(q;q)_{\infty}^4(q^{10};q^{10})_{\infty}^2},\\
L_1^{(\beta)} &= \frac{(q^{10};q^{10})_{\infty}^2}{(q^{5};q^{5})_{\infty}^3}\sum_{n=0}^{\infty} \beta(5n+4)q^{n+1} = 5q\frac{(q;q)_{\infty}^2(q^{10};q^{10})_{\infty}^4}{(q^2;q^2)_{\infty}^4(q^{5};q^{5})_{\infty}^2}.
\end{align*}  Consider the Atkin--Lehner involution
\begin{align*}
V =\begin{pmatrix}
2 & 1\\
10 & 6
\end{pmatrix}.
\end{align*}  Garvan and Morrow were able to show that for all $m\ge 1$,
\begin{align*}
L_m^{(\alpha)}\left( V\tau \right) = L_m^{(\beta)}.
\end{align*}  Once again, the functions $L_m^{(\alpha)}$, $L_m^{(\beta)}$ have identical representations in terms of their respective Hauptmoduln.  That is, the polynomial coefficients \textit{match}.  We have only to substitute one variable function for the other.

We see again that, as with Theorems \ref{Thm1} and \ref{Thm12}, the two congruence families given by Theorems \ref{ThmMG1} and \ref{ThmMG2} are equivalent: having either family allows us to demonstrate the existence of the other by applying the involution $V$.

The interest here is that the strange interrelationships between divisibility properties that we have discovered for $c\phi_{2}$ and $c\psi_{2}$, or indeed for $\phi_2$ and $\psi_2$ as in the prequel, are not at all restricted to Frobenius partitions.  Rather, they now bear close relationships to partitions with various different restrictions on even and odd parts ($\overline{\mathcal{E}\mathcal{O}}$), plane partitions (1-shell totally symmetric plane partitions), and---perhaps most dramatically---the crank parity function.

\subsection{Internal Congruences Related to Ramanujan's Theta Functions}

Yet another example of such equivalences of congruences comes from an elegant study \cite{Chern} by Chern and Tang of two families of \textit{internal} congruences modulo powers of 3 for certain quotients of Ramanujan's theta functions $\varphi(q)$, $\psi(q)$.  Let
\begin{align*}
F(\tau):=\sum_{n=0}^{\infty}ph_3(n)q^n :&= \frac{\varphi(-q^3)}{\varphi(-q)},\\
G(\tau):=\sum_{n=0}^{\infty}ps_3(n)q^n :&= \frac{\psi(q^3)}{\psi(q)},
\end{align*} wherein 
\begin{align*}
\varphi(-q) :&= \frac{(q;q)_{\infty}^2}{(q^2;q^2)_{\infty}},\\
\psi(q) :&=\frac{(q^2;q^2)^2_{\infty}}{(q;q)_{\infty}}.
\end{align*} (see Section \ref{sectinvol} below for further discussion especially on $\varphi(q)$).

Chern and Tang were able to prove the following internal congruence families:

\begin{theorem}\label{ct1}
For all $\alpha\ge 1$, we have
\begin{align}
ph_3\left(3^{2\alpha-1}n\right) \equiv ph_3\left(3^{2\alpha+1}n\right)\pmod{3^{\alpha+2}},
\end{align}
\end{theorem}
\begin{theorem}\label{ct2}
For all $\alpha\ge 1$, we have
\begin{align}
ps_3\left(3^{2\alpha-1}n + \frac{3^{2\alpha}-1}{4}\right) \equiv ps_3\left(3^{2\alpha+1}n+\frac{3^{2\alpha+2}-1}{4}\right)\pmod{3^{\alpha+2}}.
\end{align}
\end{theorem}

The opening procedure in their approach is to construct two sequences of functions $\left( \Phi_{\alpha} \right)_{\alpha\ge 1}$, $\left(\Psi_{\alpha}\right)_{\alpha\ge 1}$, which are modular over $\Gamma_0(18)$, such that

\begin{align*}
\Phi_{2\alpha-1} :&= \frac{1}{F(q^3)}\sum_{n\ge 0}ph_3(3^{2\alpha-1}n)q^n,\\
\Phi_{2\alpha} :&= \frac{1}{F(q)}\sum_{n\ge 0}ph_3(3^{2\alpha+1}n)q^n,
\end{align*}
\begin{align*}
\Psi_{2\alpha-1} :&= \frac{1}{G(q^3)}\sum_{n\ge 0}ps_3\left(3^{2\alpha-1}n + \frac{3^{2\alpha}-1}{4}\right)q^n,\\
\Psi_{\alpha} :&= \frac{1}{G(q)}\sum_{n\ge 0}ps_3\left(3^{2\alpha}n + \frac{3^{2\alpha}-1}{4}\right)q^n.
\end{align*}

One then constructs $\left( \hat{\Phi}_{\alpha} \right)_{\alpha\ge 1}$, $\left(\hat{\Psi}_{\alpha}\right)_{\alpha\ge 1}$ such that

\begin{align*}
\hat{\Phi}_{\alpha} :&= \Phi_{2\alpha+1}-\Phi_{2\alpha-1},\\
\hat{\Psi}_{\alpha} :&= \Psi_{2\alpha+1}-\Psi_{2\alpha-1}.
\end{align*}  Of course, Theorems \ref{ct1} and \ref{ct2} are both implied from showing that 
\begin{align*}
\hat{\Phi}_{\alpha} \equiv 0\pmod{3^{\alpha+2}},\\
\hat{\Psi}_{\alpha} \equiv 0\pmod{3^{\alpha+2}}.
\end{align*}  Beginning with the proof of Theorem \ref{ct1}, Chern and Tang show that each $\hat{\Phi}_{\alpha}$ is expressible as a polynomial in a certain Hauptmodul $\xi$ which lives at the cusp $[0]$ (the corresponding modular curve has genus 0), with integer coefficients divisible by $3^{\alpha+2}$.  They then show by a recursive process that $\hat{\Psi}_{\alpha}$ is represented with the same polynomials, except that the Hauptmodul $\xi$ is replaced with another, $\zeta$, which lives at the cusp $[1/2]$.

As in the previous examples of this section, the phenomenon of matching coefficients for the polynomial representations of $\hat{\Phi}_{\alpha}$ and $\hat{\Psi}_{\alpha}$ strongly suggests that an approach analogous to our work on $c\psi_{2,\beta}(n)$ should be applicable.  Indeed, let us take the Atkin--Lehner involution
\begin{align*}
V =\begin{pmatrix}
2 & -1\\
54 & -26
\end{pmatrix}.
\end{align*}  Through a straightforward manipulation of the eta transformation formulae, it can be shown that for
\begin{align}
\gamma := \frac{F(\tau)}{F(9\tau)},\\
\delta := \frac{G(\tau)}{G(9\tau)},
\end{align} we have
\begin{align}
\gamma\left(V\tau\right) = q^{-2}\delta(\tau).
\end{align}  Since Chern and Tang show that
\begin{align}
\Phi_{1} &= U_3\left( \gamma \right),\\
\Psi_{1} &= U_3\left( \delta \right),
\end{align} and since the $U_3$ operator commutes with the action of $V$, we have
\begin{align}
\Psi_{1} &= U_3\left( \delta \right)\\
&= U_3\left( \gamma(V\tau) \right)\\
&= U_3\left( \gamma \right)(V\tau)\\
&= \Phi_ {1}(V\tau).
\end{align}  In similar manner, one can prove that for every $\alpha\ge 1$, we have
\begin{align}
\Psi_{\alpha} &= \Phi_ {\alpha}(V\tau).
\end{align}

On the other hand, the Hauptmoduln $\xi$ and $\zeta$ are described in \cite{Chern} as
\begin{align*}
\xi := \frac{\varphi\left(-q^9\right)}{\varphi\left(-q\right)} = \frac{(q^2;q^2)_{\infty}(q^9;q^9)_{\infty}^{2}}{(q;q)_{\infty}^{2}(q^{18};q^{18})_{\infty}},\\
\zeta := q\frac{\psi\left( q^9 \right)}{\psi\left( q \right)} = \frac{(q;q)_{\infty}(q^{18};q^{18})_{\infty}^2}{(q^2;q^2)_{\infty}^2(q^9;q^9)_{\infty}},
\end{align*} and it can be shown that $\xi(V\tau) = \zeta(\tau)$.

Since sending $\tau\mapsto V\tau$ is linear, this preserves the associated coefficients.

One interesting difference between the latter two examples of equivalence between congruences and those of Section \ref{subsecotherfrob} and Theorems \ref{Thm1} and \ref{Thm12} is that one does not construct an explicit transformation through a congruence subgroup, e.g., an element of $\Gamma_0(10)$ or $\Gamma_0(6)$.  Of course, for the equivalence of Theorems \ref{Thm1} and \ref{Thm12}, we do use an Atkin--Lehner involution in order to construct the corresponding element of $\Gamma_0(10)$; but for the equivalence of Theorems \ref{ThmMG1} and \ref{ThmMG2}, or for the equivalence of Theorems \ref{ct1} and \ref{ct2}, we can use the associated involution alone.

\subsection{Summary}

The multiple examples above, coupled with their variety, indicate that our approach to $c\psi_{2,\beta}(n)$ is not an isolated example.  This way of looking at congruence properties associated with modular curves of composite level appears to be extraordinarily applicable, and may give some indication as to why these congruences tend to be more difficult to understand than those associated with curves of prime level.

\section{Level Reduction Theorem}\label{sectreduc}

Let us return again to the congruences for $c\psi_{2,\beta}(n)$.  We have discovered that approaching these families via Theorem \ref{fieldfuncfix} gives us some additional and important context to ongoing research in congruences families associated with composite level via the localization method, as well as to strange phenomena that were originally noted by Paule and Radu in their proof of Theorem \ref{Thm1}.  In this section we will focus on the former issue, and we will discuss the latter in the following section.

On comparing (\ref{be1a0})-(\ref{be1a1}) with (\ref{be0a0})-(\ref{be0a1}), one discovers the following:
\begin{align}
p^{(0)}_{a}+p^{(1)}_{a}\in\mathbb{Z}[x]_{\mathcal{S}}.\label{reduceto10}
\end{align}  That is, the contribution in $y$ is annihilated.  Moreover, notice that by our choice, $x$ is not only modular over $\Gamma_0(20)$, but indeed over $\Gamma_0(10)$.  So not only does the addition of our module generators annihilate the $y$ terms, it actually reduces the functions entirely to a curve of lower level.

If we combine (\ref{reduceto10}) with Theorem \ref{fieldfuncfix}, and remember from (\ref{trep}) that $t$ can be expressed as a member of $\mathbb{Z}[x]_{\mathcal{S}}$, we immediately have the following:

\begin{theorem}\label{summod5}
Let
\begin{align}
L_{\alpha} := L_{\alpha}^{(0)} + L_{\alpha}^{(1)},
\end{align}  Then
\begin{align*}
L_{\alpha}\in\mathcal{M}\left( \mathrm{X}_0(10) \right).
\end{align*} and
\begin{align*}
L_{\alpha}\equiv 0\pmod{5^{\alpha}}.
\end{align*}
\end{theorem}

While the methods of localization described above have failed to resolve Theorems \ref{Thm1} and \ref{Thm12}, it has been used successfully on congruence families associated with $\mathrm{X}_0(10)$.  As such, Theorem \ref{summod5} implies that localization may be able to prove the truth of one congruence family relative to the other.  That is, localization is not sufficient to prove Theorem \ref{Thm1}; however, it \textit{may} be able to prove Theorem \ref{Thm1} \textit{provided} that Theorem \ref{Thm12} is true, and vice versa.

We are as of yet unable to immediately prove this application of localization, though we believe that it is accessible.

\begin{conjecture}\label{conjB}
Theorem \ref{summod5} is accessible via the localization method.
\end{conjecture}

\section{The Atkin--Lehner Involution}\label{sectinvol}

We close our report with a discussion of the dual relationship our congruence pair has with respect to Ramanujan's classic $\varphi(q)$ function via a certain Atkin--Lehner involution.  This partly answers (or at least extends) some of the questions which Paule and Radu posed in their proof of Theorem \ref{Thm1}.  Recall the Atkin--Lehner involution from Section \ref{proofAtkin1} which arises from the action of the matrix
\begin{align*}
W := \begin{pmatrix}
4 & -1\\
100 & -24
\end{pmatrix}
\end{align*} on our variable $\tau\in\mathbb{H}$.

\begin{theorem}\label{lehnerxyz}
\begin{align*}
\text{Let } f = y^j \frac{x^m}{(1+5x)^n}\in\mathbb{Z}[x]_{\mathcal{S}}+y\mathbb{Z}[x]_{\mathcal{S}} \text{, for $j\in\{0,1\},$ and $m,n\in\mathbb{Z}_{\ge 0}$.  Then }
\end{align*}
\begin{align*}
f(W\tau) = (-1)^{m+n} 4^{n-2j-m}\frac{\left( 1+4x \right)^j\left(1 + 10x + 20y\right)^j \left(1 + 10x + 20x^2 + 4y\right)^m \left(1 + 5x + 5y\right)^n}{(1+5x)^{3j+2m+n}}.
\end{align*}
\end{theorem}

\begin{proof}
The functions $x$ and $y$ are modular over $\Gamma_0(20)$.  Each is itself an eta quotient, whose forms we repeat here:
\begin{align*}
x &= q\frac{(q^2;q^2)_{\infty}(q^{10};q^{10})_{\infty}^3}{(q;q)_{\infty}^3(q^5;q^5)_{\infty}} = \frac{\eta(2\tau)\eta(10\tau)^3}{\eta(\tau)^3\eta(5\tau)},\\
y &= q^2\frac{(q^2;q^2)_{\infty}^2(q^4;q^4)_{\infty}(q^5;q^5)_{\infty}(q^{20};q^{20})_{\infty}^3}{(q;q)_{\infty}^5(q^{10};q^{10})_{\infty}^2} = \frac{\eta(2\tau)^2 \eta(4\tau)\eta(5\tau)\eta(20\tau)^3}{\eta(\tau)^5\eta(10\tau)^2}.
\end{align*}  Also note that $1+5x$ is an eta quotient, since we have
\begin{align}
1+5x = \frac{(q^2;q^2)_{\infty}^5 (q^5;q^5)_{\infty}}{(q;q)_{\infty}^5(q^{10};q^{10})_{\infty}}
= \frac{\eta(2\tau)^5\eta(5\tau)}{\eta(\tau)^5\eta(10\tau)}.
\end{align}  Therefore, consider the action of $W$ on $\eta\left(\delta\tau\right)$ for $\delta|20$.  We have
\begin{align}
\delta W(\tau) = \frac{4\delta\tau  -\delta}{100\tau -24}= 
\begin{pmatrix}
4\delta & -\delta \\
100 & -24
\end{pmatrix}(\tau).
\end{align}  We want to express this in terms of the action of $\mathrm{SL}\left( 2,\mathbb{Z} \right)$.  For example, we have
\begin{align}
W(\tau) &= \frac{4\tau  -1}{100\tau -24} = 
\begin{pmatrix}
1 & 0\\
25 & 1
\end{pmatrix}(4\tau-1),\\
2W(\tau) &= \frac{8\tau  -2}{100\tau -24}= 
\begin{pmatrix}
2 & 1\\
25 & 13
\end{pmatrix}(2\tau-1),\\
4W(\tau) &= \frac{16\tau  -4}{100\tau -24}= 
\begin{pmatrix}
4 & -1\\
25 & -6
\end{pmatrix}(\tau),\\
5W(\tau) &= \frac{20\tau  -5}{100\tau -24}= 
\begin{pmatrix}
1 & 0\\
5 & 1
\end{pmatrix}(20\tau-5),\\
10 W(\tau) &= \frac{40\tau  -10}{100\tau -24}= 
\begin{pmatrix}
2 & 1\\
5 & 3
\end{pmatrix}(10\tau-3),\\
20 W(\tau) &= \frac{80\tau  -20}{100\tau -24}= 
\begin{pmatrix}
4 & -1\\
5 & -1
\end{pmatrix}(5\tau-1).
\end{align}  If we consider, say, the involution applied to $x$, we have

\begin{align}
x\left( W\tau \right) &= \frac{\eta(2W\tau)\eta(10W\tau)^3}{\eta(W\tau)^3\eta(5W\tau)}\\
&= \frac{\eta\left(\begin{pmatrix}
2 & 1\\
25 & 13
\end{pmatrix}(2\tau-1)\right)\eta\left(\begin{pmatrix}
2 & 1\\
5 & 3
\end{pmatrix}(10\tau-3)\right)^3}{\eta\left(\begin{pmatrix}
1 & 0\\
25 & 1
\end{pmatrix}(4\tau-1)\right)^3\eta\left(\begin{pmatrix}
1 & 0\\
5 & 1
\end{pmatrix}(20\tau-5)\right)}\\
&= \frac{\left(-i\left(50\tau-12\right)\right)^{1/2}\left(-i\left(50\tau-12\right)\right)^{3/2}}{\left(-i\left(100\tau-24\right)\right)^{3/2}\left(-i\left(100\tau-24\right)\right)^{1/2}}\cdot\frac{\epsilon(2,1,25,13)\epsilon(2,1,5,3)^{3}}{\epsilon(1,0,25,1)^{3}\epsilon(1,0,5,1)}\cdot\frac{\eta(2\tau-1)\eta(10\tau-3)^3}{\eta(4\tau-1)^3\eta(20\tau-5)}\\
&= \frac{1}{4}\frac{\epsilon(2,1,25,13)\epsilon(2,1,5,3)^{3}}{\epsilon(1,0,25,1)^{3}\epsilon(1,0,5,1)}\cdot\exp\left( \frac{2\pi i}{12}\left( 2\tau-1+3(10\tau-3)-3(4\tau-1)-20\tau+5 \right) \right)\\
&\times\frac{(q^2;q^2)_{\infty}(q^{10};q^{10})_{\infty}^3}{(q^4;q^4)_{\infty}^3(q^{20};q^{20})_{\infty}}\\
&=-\frac{1}{4}\frac{(q^2;q^2)_{\infty}(q^{10};q^{10})_{\infty}^3}{(q^4;q^4)_{\infty}^3(q^{20};q^{20})_{\infty}}.
\end{align}  Notice that $x(W\tau)$ is modular over $\Gamma_0(20)$.  We next note that 

\begin{align}
(1+5x)^2 x(W\tau)\in\mathcal{M}^0\left( \mathrm{X}_0(20) \right) = \mathbb{Z}[x]+y\mathbb{Z}[x].
\end{align}  We can then compute that

\begin{align}
(1+5x)^2 x(W\tau) = -(1+10x+20x^2+4y).
\end{align}  In a similar manner, we can show that

\begin{align}
(1+5x)^3 y(W\tau) = (1+4x)(1+10x+20y),\\
(1+5x)\cdot\frac{1}{1+5x(W\tau)} = (1+5x+5y).
\end{align}

\end{proof}

We thus have a means of studying the effect of $W$ on our function spaces.  Now, Paule and Radu discovered that

\begin{align}
\frac{(q^{10};q^{10})^5_{\infty}}{(q^{5};q^{5})^2_{\infty}}L_{2\alpha-1}^{(1)}(W\tau) &= \sum_{n = 0}^{\infty}c_{2\alpha-1}^{(1)}(4n)q^{4n} + q\sum_{n = 0}^{\infty}c_{2\alpha-1}^{(1)}(4n+1)q^{4n},\\
\frac{(q^{2};q^{2})^5_{\infty}}{(q;q)^2_{\infty}}L_{2\alpha}^{(1)}(W\tau) &= \sum_{n = 0}^{\infty}c_{2\alpha}^{(1)}(4n)q^{4n} + q\sum_{n = 0}^{\infty}c_{2\alpha}^{(1)}(4n+1)q^{4n}.
\end{align}  This played a vital role in recognition of the module generators $p^{(1)}_{0}, p^{(1)}_{1}$.

Upon becoming aware of the congruence family of Theorem \ref{Thm12}, we attempted to find similar properties for the functions $L_{\alpha}^{(0)}$.  After some experimentation, which included derivation of Theorem \ref{lehnerxyz} above, we discovered that

\begin{align}
\frac{(q^{5};q^{5})^2_{\infty}}{(q^{10};q^{10})_{\infty}}L_{2\alpha-1}^{(0)}(W\tau) &= \sum_{n = 0}^{\infty}c_{2\alpha-1}^{(0)}(4n)q^{4n} + q\sum_{n = 0}^{\infty}c_{2\alpha-1}^{(0)}(4n+1)q^{4n},\\
\frac{(q;q)^2_{\infty}}{(q^2;q^2)_{\infty}}L_{2\alpha}^{(0)}(W\tau) &= \sum_{n = 0}^{\infty}c_{2\alpha}^{(0)}(4n)q^{4n} + q\sum_{n = 0}^{\infty}c_{2\alpha}^{(0)}(4n+1)q^{4n}.
\end{align}  The second author of this paper recognized the pattern between these relationships.  Notice that multiplication or division by any powers of $(q^4;q^4)_{\infty}$ will not affect the behavior of the powers of $q$ that Paule and Radu noticed.  With this in mind, we recall Ramanujan's phi function, effectively the theta function
\begin{align}
\varphi(q) := \frac{(q^{2};q^{2})^5_{\infty}}{(q;q)^2_{\infty}(q^4;q^4)_{\infty}^2} = \sum_{n=-\infty}^{\infty}q^{n^2}.
\end{align}  Notice that
\begin{align}
\varphi(-q) := \frac{(q;q)^2_{\infty}}{(q^2;q^2)_{\infty}}.
\end{align}  With these, we have the following result:

\begin{theorem}\label{sellersthm}
\begin{align}
\varphi\left((-1)^{\beta+1}q^{1+4a}\right) L_{\alpha}^{(\beta)}(W\tau) \in\mathbb{Z}[[q^4]]+q\mathbb{Z}[[q^4]].
\end{align}
\end{theorem}

\begin{proof}
This can be demonstrated by examining the $q$ subsequences in the progressions $4n+2$, $4n+3$ and showing that they vanish.  To do this, we first note (as Paule and Radu did) that progressions other than $4n$ all vanish for $t(W\tau)$.  We need only concern ourselves, then, with the effect of $W$ on $p^{(\beta)}_a$.

Beginning with $a=1$, we apply $W$ to $p^{(\beta)}_1$ by Theorem \ref{lehnerxyz} above.  This gives us a combination of eta quotients.  We multiply each eta quotient by $\varphi\left( (-1)^{\beta+1}q^{5} \right)$.  We no longer have modular functions over $\Gamma_0(20)$, but that is immaterial to us.  We next examine the progressions $4n+2, 4n+3$ that each eta quotient will contribute (providing of course for the coefficients of each eta quotient and its leading power of $q$).  Finally, we show that both progressions vanish.  This is done in our Mathematica supplement (see the end of Section 1), via a careful manipulation of the Ramanujan--Kolberg algorithm which was originally designed by Radu \cite{Radu} and implemented in \cite{Smoot1}.

In the case of even-indexed $\alpha$, we need only note that
\begin{align}
\varphi\left((-1)^{\beta+1}q\right) L_{2\alpha}^{(\beta)}(W\tau) &= U^{(\beta)}_1\left( \varphi\left((-1)^{\beta+1}q^{5}\right) L_{2\alpha-1}^{(\beta)}(W\tau) \right)\\
&= U_5\left( \varphi\left((-1)^{\beta+1}q^{5}\right) L_{2\alpha-1}^{(\beta)}(W\tau) \right).
\end{align}  Of course, $U_5$ extracts the coefficients of the form $5n$.  But $5n\equiv 2,3\pmod{4}$ if and only if $n\equiv 2,3\pmod{4}$.  Thus, these progressions must also vanish.
\end{proof}

It is certainly possible that a simpler proof may yet be found.  Indeed, in a similar vein we produce the following conjecture:

\begin{conjecture}\label{sellersthmA}
\begin{align}
\sum_{\beta=0,1}\varphi\left((-1)^{\beta+1}q^{1+4a}\right) L_{\alpha}^{(\beta)}(W\tau) \in\mathbb{Z}[[q^4]].
\end{align}
\end{conjecture}  This, we are certain, can be established by a similarly careful manipulation of the progressions for $4n+1$ in each of the series.  We have only stopped for want of our time to other immediate demands.

\section{Further Work}\label{sectfin}

To the interested reader we emphasize five considerations which we deem important for future work.

\begin{enumerate}
\item Is our system of congruences for $c\psi_{2,\beta}(n)$ complete for $\beta=0,1$?  That is to say, is this pair part of a broader triple or quadruple of families all closely related via their proofs or their behavior through some associated Atkin--Lehner involutions?  Are there perhaps other mappings similar to $\sigma$ which would relate these families to others?  Indeed, we believe that this specific pair of congruence families is not so closely related to any others in the same way that they are to each other---but we cannot be certain.
\item What are the implications of identifying different congruence families with one another via homomorphisms which fix certain sub-module function spaces?
\item Do similar systems of congruence families exist for more general $c\psi_{k,\beta}(n)$, especially by holding a given $k$ fixed and varying $\beta$?
\item Is Conjecture \ref{conjB} accessible?  What implications does this pose for other congruence families, especially those associated with modular curves of large cusp count?
\item What are we to think of the beautiful yet utterly bizarre occurence of Ramanujan's phi function in Theorem \ref{sellersthm} which connects Theorems \ref{Thm1} and \ref{Thm12} via the chosen Atkin--Lehner involution?  Can this be generalized to other congruence families, again with emphasis on those associated with curves of large cusp count?  What of Conjecture \ref{sellersthmA}?
\end{enumerate}

\section{Acknowledgments}

This research was funded in part by the Austrian Science Fund (FWF) Principal Investigator Project 10.55776/PAT6428623, ``Towards a Unified Theory of Partition Congruences."  For open access purposes, the authors have applied a CC BY public copyright license to any author-accepted manuscript version arising from this submission.

In particular, the third author, Nicolas Allen Smoot, is the Principal Investigator of the aforementioned FWF Project, and he would like to thank the Austrian Government and People for their generous support.

All three authors would like to thank the anonymous referees, who provided careful and constructive advice which substantially improved this manuscript.


\begin{thebibliography}{X}

\bibitem{Andrews} G.E. Andrews, ``Generalized Frobenius Partitions," \textit{Memoirs of the American Mathematical Society} 49 (301) (1984).

\bibitem{Andrews2} G. E. Andrews, ``Integer Partitions With Even Parts Below Odd Parts and the Mock Theta Functions,"  \textit{Ann. Comb.} 22 (3), pp. 433–445 (2018).

\bibitem{Atkin} A.O.L. Atkin, ``Proof of a Conjecture of Ramanujan,'' \textit{Glasgow Mathematical Journal} 8, pp. 14-32 (1967).

\bibitem{Baner7} K. Banerjee and N.A. Smoot, ``2-Elongated Plane Partitions and Powers of 7: The Localization Method Applied to a Genus 1 Congruence Family" (submitted) (2023), \url{https://arxiv.org/abs/2306.15594}.

\bibitem{Chen2} D. Chen, R. Chen, and F.G. Garvan, ``Congruences Modulo Powers of 5 and 7 for the Crank and Rank Parity Functions and Related Mock Theta Functions," (2024).

\bibitem{Chern2} S. Chern, ``1-Shell Totally Symmetric Plane Partitions (TSPPs) Modulo Powers of 5," \textit{Ramanujan Journal} 55 (2), pp. 713-731 (2021).

\bibitem{Chern} S. Chern, D. Tang, ``Ramanujan's Theta Functions and Internal Congruences Modulo Arbitrary Powers of 3," \url{https://arxiv.org/abs/2309.06689} (2023).

\bibitem{Choi} D. Choi, S.-Y. Kang, and J. Lovejoy, ``Partitions Weighted by the Parity of the Crank," \textit{J. Combin. Theory Ser. A} 116 (5), pp. 1034-1046 (2009).

\bibitem{Diamond} F. Diamond, J. Shurman, \textit{A First Course in Modular Forms}, 4th Printing., Springer Publishing (2016).

\bibitem{Drake} B. Drake, ``Limits of Areas Under Lattice Paths," \textit{Discrete Mathematics} 309 (12) (2009).

\bibitem{Eckland} K. J. Eckland and J. A. Sellers, ``Elementary Proofs of Congruences for Drake's Variant of 2-Colored Generalized Frobenius Partitions" (submitted).

\bibitem{GarvanM} F.G. Garvan and C. Morrow, ``Congruences Modulo Powers of $5$  for Andrews’s $\overline{\mathcal{EO}}(n)$ Partition Function," preprint (2024).

\bibitem{Hirschhorn} M. D. Hirschhorn and J. A. Sellers, ``Arithmetic Properties of 1-Shell Totally Symmetric Plane Partitions," \textit{Bull. Aust. Math. Soc.} 89 (3), pp. 473-478 (2014).

\bibitem{Jiang} Y. Jiang, L. Rolen, M. Woodbury, ``Generalized Frobenius Partitions, Motzkin Paths, and Jacobi Forms," \textit{Journal of Combinatorial Theory, Series A} 190 (2022).

\bibitem{Knopp} M. Knopp, \textit{Modular Functions in Analytic Number Theory}, 2nd Ed., AMS Chelsea Publishing (1993).

\bibitem{Paule}  P. Paule, S. Radu, ``The Andrews--Sellers Family of Partition Congruences," \textit{Advances in Mathematics} 230 (2012).

\bibitem{Paule2}  P. Paule, S. Radu, ``A Proof of the Weierstraß Gap Theorem not Using the Riemann--Roch Formula," \textit{Annals of Combinatorics} 23, pp. 963-1007 (2019).

\bibitem{Radu} S. Radu, ``An Algorithmic Approach to Ramanujan--Kolberg Identities," \textit{Journal of Symbolic Computation}, 68, pp. 225-253 (2015).

\bibitem{Ramanujan} S. Ramanujan, ``Some Properties of $p(n)$, the Number of Partitions of $n$", \textit{Proceedings of the Cambridge Philosophical Society} 19, pp. 207-210 (1919).

\bibitem{Smoot} N.A. Smoot, ``A Congruence Family For 2-Elongated Plane Partitions: An Application of the Localization Method,'' \textit{Journal of Number Theory} 242, pp. 112-153 (2023).

\bibitem{Smoot0} N.A. Smoot, ``A Single-Variable Proof of the Omega SPT Congruence Family Over Powers of 5," \textit{Ramanujan Journal} 62, pp. 1-45 (2023).

\bibitem{Smoot1} N.A. Smoot, ``On the Computation of Identities Relating Partition Numbers in Arithmetic Progressions with Eta Quotients: An Implementation of Radu's Algorithm,'' \textit{Journal of Symbolic Computation} 104, pp. 276-311 (2021).

\bibitem{Watson} G.N. Watson, ``Ramanujans Vermutung über Zerfallungsanzahlen," \textit{J. Reine Angew. Math.} 179, pp. 97-128 (1938).

\end{thebibliography}
\end{document}